\documentclass{amsart}
\usepackage[twoside,lmargin=9em,rmargin=9em,top=4cm,bottom=4cm]{geometry}
\usepackage[dvipdfmx]{graphicx,xcolor}

\usepackage{amsmath}
\usepackage{amssymb}
\usepackage{amsthm}
\usepackage[all]{xy}

\newtheorem{thm}{Theorem}[section]

\newtheorem{lemma}[thm]{Lemma}

\def\CC{\mathord{\mathbb{C}}}

\def\ZZ{\mathord{\mathbb{Z}}}

\def\ket#1{\vert #1 \rangle}
\def\bra#1{\langle #1 \vert}

\newcommand{\bmm}[1]{#1}

\title{On flagged $K$-theoretic symmetric polynomials}

\author[S. Iwao]{Shinsuke Iwao}
\address[S. Iwao]{Faculty of Business and Commerce, Keio University, Hiyosi 4-1-1, Kohoku-ku, Yokohama, Kanagawa 223-8521, Japan, ORCID: https://orcid.org/0000-0002-6847-7433 }
\email{iwao-s@keio.jp}
\subjclass{05E05, 05E14, 13M10, 14N15}

\keywords{\textit{Boson-Fermion correspondence, Grothendieck polynomial, $K$-theoretic symmetric polynomial}}         

\begin{document}

\maketitle

\begin{abstract}      
We provide a fermionic description of flagged skew Grothendieck polynomials, which can be seen as a $K$-theoretic counterpart of flagged skew Schur polynomials.
Our proof relies on the Jacobi-Trudi type formula established by Matsumura.
This result generalizes the author's previous works on a fermionic description of skew Grothendieck polynomials and multi-Schur functions.
\end{abstract}

\section{Introduction}

\subsection{Overview}

Grothendieck polynomials \cite{lascoux1982structure} are a family of polynomials that represent the structure sheaf of a Schubert variety in the $K$-theory of the flag variety.
Since each Schubert variety is naturally associated with a permutation, Grothendieck polynomials are indexed by a permutation.

A flagged Grothendieck polynomial~\cite{knutson2009grobner} is a Grothendieck polynomial that associates with a vexillary permutation.
Being a $K$-theoretic analog of flagged Schur polynomials,
flagged Grothendieck polynomials possess various interesting combinatorial and algebraic properties.
Knuston-Miler-Yong~\cite{knutson2009grobner} showed that flagged Grothendieck polynomials coincide with generating functions of flagged set-valued tableaux.
Hudson-Matsumura~\cite{hudson2018vexillary} provided a Jacobi-Trudi type formula for these polynomials.

For a permutation $w\in S_n$, the inversion set (see \cite{lascoux1985schubert,wachs1985flagged}) of $w$ is the subset $I_i(w)=\{j\,|\,i<j \mbox{ and } w(i)>w(j)\}$ of $\{1,2,\dots,n\}$.
A permutation $w$ is called \textit{vexillary} if the family $\{I_i(w)\}_{i=1,2,\dots,n}$ forms a chain by inclusion.
For a vexillary permutation $w$, we associate a partition $\lambda(w)$ by arranging the cardinalities of inversion sets in decreasing order.
A flagging of $w$ is the sequence obtained by arranging $\min I_i(w)-1$ in increasing order.
The flagged Grothendieck polynomial $G_w(x)$ is often written as $G_{\lambda,f}(x)$, where $\lambda=\lambda(w)$ and $f$ is the flagging of $w$.

In the work~\cite{matsumura2019flagged}, Matsumura generalized these polynomials for skew shapes $\lambda/\mu$ with a flagging $f/g$, where $f=(f_1,\dots,f_r)$ and $g=(g_1,\dots,g_r)$ are sequences of natural numbers.
He proved that a flagged skew Grothendieck polynomial, which was defined as a generating function of flagged skew set-valued tableaux, admits a Jacobi-Trudi type formula.
For $n,p,q\in \ZZ$, define $G_n^{[p/q]}(x)$ by the generating function
\begin{equation}
\sum_{n\in \ZZ}G_n^{[p/q]}(x)z^n=
\begin{cases}
\frac{1}{1+\beta u^{-1}}\prod_{k=q}^p\frac{1+\beta x_k}{1-x_ku} & (p\geq  q)\\
\frac{1}{1+\beta u^{-1}} & (p<  q)
\end{cases}.
\end{equation}
Then, Matsumura's determinant formula \cite[\S 4]{matsumura2019flagged} is expressed as 
\begin{equation}
G_{\lambda/\mu,f/g}(x)=
\det\left(
\sum_{s=0}^\infty
\binom{i-j}{s}\beta^s G^{[f_i/g_j]}_{\lambda_i-\mu_j-i+j+s}(x)
\right).
\end{equation}

In this paper, we study such ``flagged skew Grothendieck polynomials''.
However, we should adopt the slightly different definition
\begin{equation}\label{eq:gen1}
\sum_{n\in \ZZ}G_n^{[[p/q]]}(x)z^n=
\begin{cases}
\frac{1}{1+\beta u^{-1}}
\prod_{k=q}^p\frac{1+\beta x_k}{1-x_ku} & (p\geq  q)\\
\frac{1}{1+\beta u^{-1}} & (p=q-1)\\
\frac{1}{1+\beta u^{-1}} 
\prod_{k=p+1}^{q-1}\frac{1-x_ku}{1+\beta x_k}
& (p<  q-1)
\end{cases}
\end{equation}
and consider the polynomial
\begin{equation}\label{eq:def_det}
G_{\lambda/\mu,[[f/g]]}(x)=
\det\left(
\sum_{s=0}^\infty
\binom{i-j}{s}\beta^s G^{[[f_i/g_j]]}_{\lambda_i-\mu_j-i+j+s}(x)
\right)
\end{equation}
because the definition \eqref{eq:gen1} is more suitable for our algebraic calculations.
The polynomials $G_{\lambda/\mu,f/g}(x)$ and $G_{\lambda/\mu,[[f/g]]}(x)$ are different in general but coincide with each other if $f_i+\lambda_i-i\geq g_j+\mu_j-j$ whenever $f_i<g_j-1$.
In particular, if $g_1=g_2=\dots=g_r=1$, they coincide with each other for any skew shape $\lambda/\mu$ and a flagging $f$.

Our aim is to construct a new algebraic description of $G_{\lambda/\mu,[[f/g]]}(x)$ using concepts from vertex algebras, including fermion Fock spaces, vertex operators, and vacuum expectation values.
In the previous work~\cite[\S 4]{iwao2022free}, the author of the paper presented a fermionic description of skew Grothendieck polynomials.
Generalizing this approach, we show the main theorem (Theorem \ref{thm:main}) which provides a fermionic description for the flagged Grothendieck polynomial.
We emphasize the fact that our descriptions \eqref{eq:main_expression}, \eqref{eq:des_Glm} are not found in the previous work \cite{iwao2023free} of multi-Schur functions.
This suggests that there is still room for generalization of fermionic descriptions of $K$-theoretic polynomials.

\subsection*{Acknowledgements}
This work was partially supported by JSPS KAKENHI Grant Numbers 19K03065, 22K03239, and 23K03056.

\section{Preliminaries}\label{sec:pre}

\subsection{Fermion Fock space}

Let $\mathcal{A}$ be the $\CC$-algebra generated by the \textit{free fermions} $\psi_n,\psi_n^\ast$ $(n\in \ZZ)$ with anti-commutation relations
\[
[\psi_m,\psi_n]_+=[\psi^\ast_m,\psi^\ast_n]_+=0,\qquad
[\psi_m,\psi^\ast_n]_+=\delta_{m,n},
\]
where $[A,B]_+=AB+BA$ is the anti-commutator.

Let $\mathcal{F}=
\mathcal{A}\cdot \ket{0}
$ be the \textit{Fock space}, the left $\mathcal{A}$-module generated by the \textit{vacuum vector} 
\[
\psi_m\ket{0}=\psi^\ast_n\ket{0}=0,\quad
m< 0,\ n\geq 0.
\]
We also use the \textit{dual Fock space} $\mathcal{F}^\ast:=\bra{0}\cdot \mathcal{A}$, the right $\mathcal{A}$-module generated by the \textit{dual vacuum vector}
\[
\bra{0}\psi_n=\bra{0}\psi^\ast_m=0,\quad 
m< 0,\ n\geq 0.
\]
There uniquely exists an anti-algebra involution on $\mathcal{A}$ 
\[
{}^\ast:\mathcal{A}\to \mathcal{A};\quad \psi_n\leftrightarrow \psi_n^\ast,
\]
satisfying $(ab)^\ast=b^\ast a^\ast$ and $(a^\ast)^\ast=a$ for $a,b\in \mathcal{A}$, which induces the $\CC$-linear involution
\[
\omega:\mathcal{F}\to {\mathcal{F}}^\ast,\quad X\ket{0}\mapsto \bra{0}{X}^\ast.
\]

The \textit{vacuum expectation value} \cite[\S 4.5]{miwa2012solitons} is the unique $\CC$-bilinear map
\begin{equation}\label{eq:vacuum_expectation_value}
{\mathcal{F}}^\ast\otimes_k\mathcal{F}\to k,\quad 
\bra{w}\otimes \ket{v}\mapsto \langle{w}\vert v\rangle,
\end{equation}
satisfying $\langle 0\vert 0\rangle=1$,
$(\bra{w}\psi_n) \ket{v}=\bra{w} (\psi_n\ket{v})$,
and
$(\bra{w}\psi_n^\ast) \ket{v}=\bra{w} (\psi_n^\ast\ket{v})$.
For any expression $X$, we write 
$\bra{w}X\ket{v}:=(\langle{w}\vert X)\ket{v}
=\bra{w}(\vert X\ket{v})
$.
The expectation value $\bra{0}X\ket{0}$ is often abbreviated as $\langle X\rangle$.


\begin{thm}[Wick's theorem
(see {\cite[\S 2]{alexandrov2013free}, \cite[Exercise 4.2]{miwa2012solitons}})
]\label{thm:Wick}
Let $\{m_1,\dots,m_r\}$ and $\{n_1,\dots,n_{r}\}$ be sets of integers.
Then we have
\[
\langle 
\psi_{m_1}\cdots\psi_{m_{r}}
\psi^\ast_{n_r}\cdots\psi^\ast_{n_{1}}
\rangle
=\det(\langle \psi_{m_i}\psi^\ast_{n_j} \rangle)_{1\leq i,j\leq r}.
\]
\end{thm}

For an integer $m$, we define the \textit{shifted vacuum vectors} $\ket{m}\in \mathcal{F}$ and $\bra{m}\in\mathcal{F}^\ast$ by
\[
\ket{m}=
\begin{cases}
\psi_{m-1}\psi_{m-2}\cdots \psi_0\ket{0}, & m\geq 0,\\
\psi^\ast_{m} \cdots\psi^\ast_{-2}\psi^\ast_{-1}\ket{0}, & m<0,
\end{cases}\qquad
\bra{m}=
\begin{cases}
\bra{0}\psi^\ast_0\psi^\ast_1\dots \psi^\ast_{m-1}, & m\geq 0,\\
\bra{0}\psi_{-1}\psi_{-2}\dots \psi_{m}, & m<0.
\end{cases}
\]


\subsection{Vertex operators and commutation relations}\label{sec:comm_rels}

For any monomial expression $M$ in $\psi_n$ and $\psi^\ast_n$, the \textit{normal ordering} 
\[
:M:\quad \in \mathcal{A}
\]
is defined by moving the \textit{annihilation operators}
\[
\psi_m,\quad \psi^\ast_n,\qquad  m<0,\, n\geq 0
\]
to the right, and multiplying $-1$ for each move
(See \cite[\S 2]{alexandrov2013free}, \cite[\S 5.2]{miwa2012solitons}).
For example, we have $:\psi_1\psi^\ast_1:=\psi_1\psi^\ast_1$ and $:\psi^\ast_1\psi_1:=-\psi_1\psi^\ast_1$.
The normal ordering extends to the $\CC$-linear map
\[
\{\mbox{polynomial expressions in $\psi_n$ and $\psi^\ast_n$ with coefficients in $\CC$}\}\to \mathcal{A};\quad
X\mapsto {:X:}
\]

Let $a_m$ ($m\in \ZZ$) be the \textit{current operator} $a_m=\sum_{k\in \ZZ} :\psi_k\psi^\ast_{k+m}:$, which satisfies
\begin{equation}\label{eq:relation_added}
[a_m,a_n]=m\delta_{m+n,0},\qquad
[a_m,\psi_n]=\psi_{n-m},\qquad
[a_m,\psi^\ast_n]=-\psi^\ast_{n+m},
\end{equation}
where $[A,B]=AB-BA$
(see \cite[\S 5.3]{miwa2012solitons}).
If $\ket{v})=\bra{v^\ast}$, we have
$
\omega(a_i\ket{v})=\bra{v^\ast}a_{-i}
$ for any $n\in \ZZ$.

Let $X=(X_1,X_2,\dots)$ be a set of (commutative) variables.
We define the \textit{Hamiltonian operator} 
\[
H(X)=\sum_{n>0}\frac{p_n(X)}{n}a_n,\qquad p_n(X)=X_1^n+X_2^n+\cdots
\]
and its dual
\[
H^\ast(X)=\sum_{n>0}\frac{p_n(X)}{n}a_{-n}.
\]
We define the \textit{vertex operators} by
\[
e^{H(X)}=\sum_{n=0}^\infty \frac{H(X)^n}{n!},\qquad
e^{H^\ast(X)}=\sum_{n=0}^\infty \frac{H^\ast(X)^n}{n!}
\]

Let
$\psi(z)=\sum_{n\in \ZZ}\psi_nz^n$ and $\psi^\ast(z)=\sum_{n\in \ZZ}\psi^\ast_nz^n$ be the fermion fields.
Here, we summarize useful commutation relations concerning the vertex operators and the fermion fields. 
For proofs, see the textbook~\cite{miwa2012solitons}.
\begin{gather}
e^{H(X)}\psi(z)e^{-H(X)}=\left(\prod_{i}\frac{1}{1-X_iz}\right)\psi(z),\label{eq:comm1}\\
e^{H^\ast(X)}\psi(z)e^{-H^\ast(X)}=\left(\prod_{i}\frac{1}{1-X_iz^{-1}}\right)\psi(z),\label{eq:comm2}\\
e^{H(X)}e^{-H^\ast(Y)}=
\left(\prod_{i,j}(1-X_iY_j)\right)
e^{-H^\ast(Y)}e^{H(X)},\label{eq:comm3}\\
\bra{-r}\psi^\ast(w)\psi(z)\ket{-r}
=\frac{z^{-r}w^{-r}}{1-zw},\label{eq:comm4}
\end{gather}
where $\frac{z^{-r}w^{-r}}{1-zw}=\sum_{p=-r}^\infty z^pw^p$.

\section{Flagged Skew Grothendieck polynomial}

\subsection{$G^{[[f/g]]}_n(x)$}

In this section, we give a fermionic presentation of flagged skew Grothendieck polynomials. 
For brevity, we adopt the convention 
\[
\prod_{k=q}^pX_k
=
\begin{cases}
X_qX_{q+1}\cdots X_p & (p\geq q)\\
1 & (p=q-1)\\
X^{-1}_{p+1}X^{-1}_{p+2}\cdots X^{-1}_{q-1} & (p<q-1)
\end{cases}
\]
for a sequence $X_1,X_2,\dots$ of (commutative) functions.

Let $x^{[f]}=(x_1,x_2,\dots,x_f)$ and 
\[
H(x^{[f/g]})=H(x^{[f]})-H(x^{[g-1]}).
\]
If $f\geq g$, $H(x^{[f/g]})$ coincides with
$
H(x_{g},x_{g+1},\dots,x_f)
$.
From (\ref{eq:comm1}--\ref{eq:comm3}), we have
\[
\begin{aligned}
e^{H(x^{[f/g]})}\psi(z)e^{-H^\ast(-\beta)}
&=
\left(
\frac{1}{1+\beta z^{-1}}\prod_{j=g}^f\frac{1+\beta x_i}{1-x_iz}
\right)
e^{-H^\ast(-\beta)}\psi(z)e^{H(x^{[f/g]})},
\end{aligned}
\]
where the rational function on the right hand side expands in the ring\footnote{Note that the two rings $\CC((z))[[\beta]]$ and $\CC[[\beta]]((z))$ are different.
In fact, $\CC((z))[[\beta]]$ contains
\[
1+\frac{\beta}{z}+\frac{\beta^2}{z^2}+\cdots,
\]
while $\CC[[\beta]]((z))$ does not.
}
\[
\CC[x_1,x_2,\dots]((z))[[\beta]].
\]

Comparing this equation with \eqref{eq:gen1}, we obtain
\begin{equation}\label{eq:comm_Groth}
e^{H(x^{[f/g]})}
\psi(z)e^{-H^\ast(-\beta)}
=
\left(
\sum_{n\in \ZZ}G_n^{[[f/g]]}(x)z^{n}
\right)
e^{-H^\ast(-\beta)}\psi(z)e^{H(x^{[f/g]})}.
\end{equation}
A similar calculation leads
\begin{equation}\label{eq:comm_dual_Groth}
e^{H^\ast(-\beta)}
\psi^\ast(w)
e^{-H(x^{[f/g]})}
=
\left(
\sum_{n\in \ZZ}G_n^{[[f/g]]}(x)w^{-n}
\right)^{-1}
e^{-H(x^{[f/g]})}
\psi^\ast(w)
e^{H^\ast(-\beta)}.
\end{equation}

\begin{lemma}
$G_n^{[[f/g]]}(x)$ has the fermionic description
\[
G_n^{[[f/g]]}(x)
=
\bra{0}
e^{H(x^{[f/g]})}
\psi_{n-1}
e^{-H^\ast(-\beta)}
\ket{-1}.
\]
\end{lemma}
\begin{proof}
Let $F_n=\bra{0}
e^{H(x^{[f/g]})}
\psi_{n-1}
e^{-H^\ast(-\beta)}
\ket{-1}$.
Since $\bra{0}e^{H^\ast(-\beta)}=\bra{0}$ and $e^{H(\bmm{x}^{[f/g]})}\ket{-1}=\ket{-1}$, we have
\[
\begin{aligned}
\sum_{n\in \ZZ}
F_n
z^{n-1}
&=
\bra{0}
e^{H(\bmm{x}^{[f/g]})}
\psi(z)
e^{-H^\ast(-\beta)}
\ket{-1}\\
&=
\left(
\sum_{m\in \ZZ}G_m^{[[f/g]]}(x)z^m
\right)
\bra{0}
e^{-H^\ast(-\beta)}
\psi(z)
e^{H(\bmm{x}^{[[f/g]]})}
\ket{-1}\qquad
\mbox{(Eq.~\eqref{eq:comm_Groth})}
\\
&
=
\left(
\sum_{m\in \ZZ}G_m^{[[f/g]]}(x)z^m
\right)
\bra{0}
\psi(z)
\ket{-1}.
\end{aligned}
\]
Since $\bra{0}\psi(z)\ket{-1}=\bra{0}\psi(z)\psi^\ast_{-1}\ket{0}=z^{-1}$, we conclude $F_n=G_n^{[[f/g]]}(x)$.
\end{proof}

\subsection{Fermionic description}

We introduce a fermionic description of skew Flagged Grothendieck polynomials in this section.
For a sequence of noncommutative elements $P_1,P_2,\dots$, denote
\[
\prod_{i:1\to N}P_i:=P_1P_2\cdots P_N,\qquad
\prod_{i:N\to 1}P_i:=P_N\cdots P_2P_1.
\]
For any $X,Y$, we use the notation $\mathrm{Ad}_{e^X}(Y)=e^XYe^{-X}$.

Let $f=(f_1, f_2, \dots ,f_r)$ and $g=(g_1, g_2, \dots , g_r)$ be sequences of positive integers.
Let $G_{\lambda/\mu,[[f/g]]}(x)$ be the polynomial defined by the determinantal formula \eqref{eq:def_det}.
The following is the main theorem of the paper:
\begin{thm}\label{thm:main}
Let $\lambda/\mu$ be a skew partition. 
Then, the flagged skew Grothendieck polynomial $G_{\lambda/\mu,[[f/g]]}(x)$ is expressed as
\begin{equation}\label{eq:main_expression}
\bra{-r}
\left(
\prod_{j:r\to 1}
\psi^\ast_{\mu_j-j}
e^{H^\ast(-\beta)}
e^{-H(\bmm{x}^{[g_j-1/g_{j-1}]})}
\right)
\left(
\prod_{i:1\to r}
e^{H(\bmm{x}^{[f_i/f_{i-1}+1]})}\psi_{\lambda_i-i}e^{-H^\ast(-\beta)}
\right)
\ket{-r},
\end{equation}
where $x^{[f_1/f_0+1]}=x^{[f_1]}$ and $x^{[g_1-1/g_0]}=x^{[g_1-1]}$.
\end{thm}
\begin{proof}
Since 
\[
H(x^{[f_i/f_{i-1}+1]})=H(x^{[f_i]})-H(x^{[f_{i-1}]}),\quad   H(\bmm{x}^{[g_j-1/g_{j-1}]})=H(\bmm{x}^{[g_j-1]})-H(\bmm{x}^{[g_{j-1}-1]}),
\]
and $e^{H(x)}\ket{-r}=\ket{-r}$,
the expectation value \eqref{eq:main_expression} equals to
\begin{equation}\label{eq:main_expression2}
\begin{split}
\bra{-r}
e^{-H(x^{[g_r-1]})}
\left(
\prod_{j:r\to 1}
e^{H(\bmm{x}^{[g_j-1]})}
\psi^\ast_{\mu_j-j}
e^{H^\ast(-\beta)}
e^{-H(\bmm{x}^{[g_j-1]})}
\right)\\
\cdot
\left(
\prod_{i:1\to r}
e^{H(\bmm{x}^{[f_i]})}\psi_{\lambda_i-i}e^{-H^\ast(-\beta)}
e^{-H(\bmm{x}^{[f_i]})}
\right)
\ket{-r}.
\end{split}
\end{equation}
By using the equations
\[
\begin{gathered}
\left(\prod_{j:r\to 1}P_j\right)e^{-rX}
=
\prod_{j:r\to 1}(e^{(r-j)X}P_je^{-(r-j+1)X})
=
\prod_{j:r\to 1}\mathrm{Ad}_{e^{(r-j)X}}(P_je^{-X}),\\
e^{rX}
\prod_{i:1\to r}P_i=
\prod_{i:1\to r}
(e^{(r-i+1)X}P_ie^{-(r-i)X})
=
\prod_{i:1\to r}\mathrm{Ad}_{e^{(r-i)X}}(e^XP_i),
\end{gathered}
\]
\eqref{eq:main_expression2} is rewritten as
\begin{equation}\label{eq:tocyuu}
\begin{split}
\bra{-r}
e^{-H(\bmm{x}^{[g_r-1]})}
\left(
\prod_{j:r\to 1}
\mathrm{Ad}_{e^{(r-j)H^\ast(-\beta)}}
\left(
e^{H(\bmm{x}^{[g_j-1]})}
\psi^\ast_{\mu_j-j}e^{H^\ast(-\beta)}
e^{-H(\bmm{x}^{[g_j-1]})}
e^{-H^\ast(-\beta)}
\right)
\right)\\
\cdot
\left(
\prod_{i:1\to r}
\mathrm{Ad}_{e^{(r-i)H^\ast(-\beta)}}
\left(
e^{H^\ast(-\beta)}
e^{H(\bmm{x}^{[f_i]})}\psi_{\lambda_i-i}e^{-H^\ast(-\beta)}
e^{-H(\bmm{x}^{[f_i]})}
\right)
\right)
\ket{-r}.
\end{split}
\end{equation}
Let
\[
\begin{gathered}
A_i(z_i):=
\mathrm{Ad}_{e^{(r-i)H^\ast(-\beta)}}
\left(
e^{H^\ast(-\beta)}
e^{H(\bmm{x}^{[f_i]})}\psi(z_i)e^{-H^\ast(-\beta)}
e^{-H(\bmm{x}^{[f_i]})}
\right),\\
B_j(w_j):=
\mathrm{Ad}_{e^{(r-j)H^\ast(-\beta)}}
\left(
e^{H(\bmm{x}^{[g_j-1]})}
\psi^\ast(w_j)e^{H^\ast(-\beta)}
e^{-H(\bmm{x}^{[g_j-1]})}
e^{-H^\ast(-\beta)}
\right)
\end{gathered}
\]
Then, by Wick's theorem (Theorem \ref{thm:Wick}), the expectation value \eqref{eq:tocyuu} equals to the coefficient of 
$z_1^{\lambda_1-1}\cdots z_r^{\lambda_r-r}\cdot 
w_1^{\mu_1-1}\cdots w_r^{\mu_r-r}
$ in the determinant
\begin{equation}\label{eq:determinant}
\det\left(
\bra{-r}
e^{-H(\bmm{x}^{[g_r-1]})}
B_j(w_j)A_i(z_i)
\ket{-r}
\right)_{i,j}.
\end{equation}
From \eqref{eq:comm_Groth}, $A_i(z_i)$ satisfies
\[
\begin{aligned}
A_i(z_i)
&=
\left(
\sum_{n\in \ZZ}G_n^{[f_i]}(x)z_i^n
\right)
\cdot 
\mathrm{Ad}_{e^{(r-i)H^\ast(-\beta)}}(\psi(z_i))\\
&=
\left(
\sum_{n\in \ZZ}G_n^{[f_i]}(x)z_i^n
\right)
\cdot (1+\beta z_i^{-1})^{-(r-i)}\cdot \psi(z_i).
\end{aligned}
\]
On the other hand, since
\[
\begin{aligned}
&
e^{H(\bmm{x}^{[g_j-1]})}
\psi^\ast(w_j)e^{H^\ast(-\beta)}
e^{-H(\bmm{x}^{[g_j-1]})}
e^{-H^\ast(-\beta)}\\
&=
(1+\beta w_j)^{-1}
\cdot 
e^{H(\bmm{x}^{[g_j-1]})}
e^{H^\ast(-\beta)}\psi^\ast(w_j)
e^{-H(\bmm{x}^{[g_j-1]})}
e^{-H^\ast(-\beta)}\\
&=
\left(
\sum_{n\in \ZZ}G_n^{[g_j-1]}(x)w_j^{-n}
\right)^{-1}
(1+\beta w_j)^{-1}
\cdot
\psi^\ast(w_j)\qquad
\mbox{(Eq.~\eqref{eq:comm_dual_Groth})}
,
\end{aligned}
\]
we obtain
\[
\begin{aligned}
B_j(w_j)
&=
\left(
\sum_{n\in \ZZ}G_n^{[g_j-1]}(x)w_j^{-n}
\right)^{-1}
\cdot 
(1+\beta w_j)^{-1}
\mathrm{Ad}_{e^{(r-j)H^\ast(-\beta)}}(\psi^\ast(w_j))\\
&=
\left(
\sum_{n\in \ZZ}G_n^{[g_j-1]}(x)w_j^{-n}
\right)^{-1}
\cdot 
(1+\beta w_j)^{r-j-1}
\cdot 
\psi^\ast(w_j).
\end{aligned}
\]
Therefore, we have
\[
\begin{aligned}
&
\bra{-r}
e^{-H(\bmm{x}^{[g_r-1]})}
B_j(w_j)A_i(z_i)
\ket{-r}
\\
&
=
\frac{\sum_n G_{n}^{[f_i]}(x)z_i^n}{\sum_n G_n^{[g_j-1]}(x)w_j^{-n}}
\frac{(1+\beta w_j)^{r-j-1}}{(1+\beta z_i^{-1})^{r-i}}
\bra{-r}
e^{-H(\bmm{x}^{[g_r-1]})}
\psi^\ast(w_j)
\psi(z_i)
\ket{-r}
\\
&=
\frac{\sum_n G_{n}^{[f_i]}(x)z_i^n}{\sum_n G_n^{[g_j-1]}(x)w_j^{-n}}
\frac{(1+\beta w_j)^{r-j-1}}{(1+\beta z_i^{-1})^{r-i}}
\frac{\prod_{k=1}^{g_r-1} (1-x_k z_i)}{\prod_{k=1}^{g_r-1}(1-x_kw_j^{-1})}
\bra{-r}
\psi^\ast(w_j)
\psi(z_i)
\ket{-r}
\\
&=
\frac{\sum_n G_{n}^{[f_i]}(x)z_i^n}{\sum_n G_n^{[g_j-1]}(x)w_j^{-n}}
\frac{(1+\beta w_j)^{r-j-1}}{(1+\beta z_i^{-1})^{r-i}}
\frac{\prod_{k=1}^{g_r-1} (1-x_k z_i)}{\prod_{k=1}^{g_r-1}(1-x_kw_j^{-1})}
\frac{z_i^{-r}w_j^{-r}}{1-z_iw_j}
\qquad 
\mbox{(Eq.~\eqref{eq:comm4})}
\\
&=
\frac{\prod_{k=1}^{f_i} (1+\beta x_k)}{\prod_{k=1}^{g_j-1} (1+\beta x_k)}
\frac{(1+\beta w_j)^{r-j}}{(1+\beta z_i^{-1})^{r-i+1}}
\frac{\prod_{k=f_i+1}^{g_r-1} (1-x_k z_i)}{\prod_{k=g_j}^{g_r-1}(1-x_kw_j^{-1})}
\frac{z_i^{-r}w_j^{-r}}{1-z_iw_j}\\
&=:F(z_i,w_j).
\end{aligned}
\]
To take the coefficient of $z^{\lambda_i-i}w^{\mu_j-j}$ in $F(z,w)$, we use the complex line integral. 
Note that the expansion of the rational function $F(z,w)$ in the field \[\CC[x_1,x_2,\dots]((w^{-1}))((z))[[\beta]]\] coincides with the Laurent expansion on the domain $\{|\beta|<|z|<|w^{-1}|<|x_k^{-1}|\,;\,\forall k\}$. 
Then, we have
\begin{align}
&
[w^{\mu_j-j}]
\bra{-r}
e^{-H(\bmm{x}^{[g_r-1]})}
B_j(w)A_i(z)
\ket{-r}
\nonumber\\
&=\frac{1}{2\pi \sqrt{-1}}
\oint_{|\beta|<|z|<|w^{-1}|<|x_k^{-1}|}
F(z,w)
\cdot (w^{-1})^{\mu_j-j}
\frac{d(w^{-1})}{(w^{-1})}\nonumber\\
&=\frac{1}{2\pi \sqrt{-1}}
\oint_{|\beta|<|z|<|t|<|x_k^{-1}|}
F(z,t^{-1})\cdot t^{\mu_j-j-1}dt.\label{eq:contour}
\end{align}
Since $F(z,t^{-1})\cdot t^{\mu_j-j-1}dt$ is expressed as
\[
\frac{\prod_{k=1}^{f_i} (1+\beta x_k)}{\prod_{k=1}^{g_j-1} (1+\beta x_k)}
\frac{(t+\beta)^{r-j}}{(1+\beta z^{-1})^{r-i+1}}
\frac{\prod_{k=f_i+1}^{g_r-1} (1-x_k z)}{\prod_{k=g_j}^{g_r-1}(1-x_kt)}
\frac{z_i^{-r}}{t-z}\cdot t^{\mu_j}dt,
\]
the contour integral \eqref{eq:contour} coincides with the residue of the differential form at $t=z$.
Finally, it equals to
\begin{equation}\label{eq:entry}
\begin{aligned}
\frac{1}{1+\beta z^{-1}}\prod_{k=g_j}^{f_i}\frac{1+\beta x_k}{1-x_kz}\cdot (z+\beta)^{i-j}z^{\mu_j-i}
=
\left(\sum_{n\in \ZZ}G_n^{[[f_i/g_j]]}(x)z^n\right)
(1+\beta z^{-1})^{i-j}z^{\mu_j-j}.
\end{aligned}
\end{equation}
Since the coefficient of $z^{\lambda_i-i}$ in \eqref{eq:entry} is 
$
\sum_{s=0}^\infty \binom{i-j}{s}\beta^s G^{[[f_i/g_j]]}_{\lambda_i-\mu_j-i+j+s}(x)
$,
the coefficient of  $z_1^{\lambda_1-1}\cdots z_r^{\lambda_r-r}\cdot 
w_1^{\mu_1-1}\cdots w_r^{\mu_r-r}
$ in the determinant \eqref{eq:determinant} is 
\[
\det\left(
\sum_{s=0}^\infty \binom{i-j}{s}\beta^s G^{[[f_i/g_j]]}_{\lambda_i-\mu_j-i+j+s}(x)
\right)_{i,j}.
\]
Comparing it with \eqref{eq:def_det}, we conclude the theorem.
\end{proof}

\subsection{Remarks}

If $g_1=g_2=\dots=g_r=1$, $G_{\lambda/\mu,[[f/g]]}(x)$ reduces to the (usual) flagged Grothendieck polynomial $G_{\lambda/\mu,f}(x)$.
In this case, our main Theorem \ref{thm:main} reduces to 
\begin{equation}\label{eq:des_Glm}
\begin{split}
G_{\lambda/\mu,f}(x)
=&
\bra{-r}
\psi^\ast_{\mu_r-r}e^{H^\ast(-\beta)}
\dots
\psi^\ast_{\mu_2-2}e^{H^\ast(-\beta)}
\psi^\ast_{\mu_1-1}e^{H^\ast(-\beta)}\\
&\cdot 
(e^{H(x^{[f_1]})}\psi_{\lambda_1-1}e^{-H^\ast(-\beta)})
\cdots
(e^{H(x^{[f_r/f_{r-1}]})}\psi_{\lambda_r-r}e^{-H^\ast(-\beta)})
\ket{-r}.
\end{split}
\end{equation}
By taking $f_1=f_2=\dots=f_r=n$, we recover the fermionic presentation of the skew Grothendieck polynomial given in \cite[\S 4.2]{iwao2022free}.
This expression is not found in the previous work \cite{iwao2023free} of fermionic presentations for multi-Schur functions.

\providecommand{\bysame}{\leavevmode\hbox to3em{\hrulefill}\thinspace}
\providecommand{\MR}{\relax\ifhmode\unskip\space\fi MR }
\providecommand{\MRhref}[2]{%
  \href{http://www.ams.org/mathscinet-getitem?mr=#1}{#2}
}
\providecommand{\href}[2]{#2}


\end{document}